\documentclass[reqno]{amsart}
\usepackage{amsmath,amsthm,amscd,amssymb,amsfonts,amsbsy}
\usepackage{latexsym}
\usepackage{pxfonts}
\usepackage{enumerate}

\numberwithin{equation}{section}

\theoremstyle{plain}
\newtheorem{theorem}[equation]{Theorem}
\newtheorem{lemma}[equation]{Lemma}

\theoremstyle{definition}

\newtheorem*{acknowledgment}{Acknowledgment}
\newtheorem*{conclusion}{Conclusion}

\theoremstyle{remark}
\newtheorem{remark}[equation]{Remark}

\newcommand{\bR}{\mathbb R}
\newcommand{\bP}{\mathbb P}

\newcommand{\bE}{\mathbb E}
\newcommand{\bH}{\mathbb H}

\newcommand{\set}[1]{\left\{#1\right\}}

\renewcommand{\epsilon}{\varepsilon}

\newcommand{\Co}{\mathcal{C}^{1,2}}

\newcommand{\Ho}{\mathcal{C}^{1+\alpha/2,2+\alpha}}

\renewcommand{\qedsymbol}{$\blacksquare$}

\begin{document}
\title[Regularity of a degenerate parabolic equation]{Regularity of a degenerate parabolic equation appearing in Ve\v{c}e\v{r}'s unified pricing of Asian options}

\author[H. Dong]{Hongjie Dong}
\address[H. Dong]{Division of Applied Mathematics, Brown University, 182 George Street, Box F, Providence, RI 02912, USA}
\email{Hongjie\_Dong@brown.edu}
%\thanks{H. Dong was partially supported by the ...}

\author[S. Kim]{Seick Kim}
\address[S. Kim]{Department of Mathematics, Yonsei University, 50 Yonsei-ro, Seodaemun-gu, Seoul 120-749, Republic of Korea}
\email{kimseick@yonsei.ac.kr}
\thanks{S. Kim is partially supported by NRF Grant No. 2012R1A1A2040411.}

%\date{Received: date / Revised version: date}

\begin{abstract}
Ve\v{c}e\v{r} \cite{Vecer02} derived a degenerate parabolic equation with a boundary condition characterizing the price of Asian options with generally sampled average.
It is well understood that there exists a unique probabilistic solution to such a problem but it remained unclear whether the probabilistic solution is a classical solution.
We prove that the probabilistic solutions to Ve\v{c}e\v{r}'s PDE are regular.
\end{abstract}

\maketitle

%------------------------------------------------------------------------------%
\section{Introduction and Main result}
%------------------------------------------------------------------------------%
An Asian option is a specialized form of an option where the payoff is not determined by the underlying price at maturity, but is connected to the average value of the underlying security over certain time interval.
In an interesting article \cite{Vecer02}, Ve\v{c}e\v{r} presented a unifying PDE approach for pricing Asian options that works for both discrete and continuous arithmetic average.
By using a dimension reduction technique, he derived a simple degenerate parabolic equation in two variables $(t,x)\in [0,T)\times\bR$
\begin{equation}
\label{eq:I03c}
u_t+\frac{1}{2}
\left(x-e^{-\int_0^t d\nu(s)}q(t) \right)^2 \sigma^2 u_{xx}=0\quad
\end{equation}
supplemented by a terminal condition
\begin{equation}
\label{eq:I04d}
u(T,x)=(x-K)_{+}:= \max(x-K,0),
\end{equation}
which gives the price of the Asian option.
Here, $\nu(t)$ is the measure representing the dividend yield, $\sigma$
is the volatility of the underlying asset, $q(t)$ is the trading strategy
given by
\[
q(t)=\exp\left\{-\int_t^T \,d\nu(s)\right\}\cdot\int_t^T
\exp\left\{-r(T-s)+\int_s^T \,d\nu(\tau)\right\}\,d\mu(s),
\]
where $r$ is the interest rate and $\mu(t)$ represents a general weighting
factor.
We ask readers to refer to \cite{Vecer02} for the derivation of
the equation \eqref{eq:I03c}.
It should be noted that the function $b$ given by
\begin{align}
\label{eq:I04x}
b(t) &:=e^{-\int_0^t d\nu(s)}q(t)\\
\nonumber
&= \exp\left\{-\int_0^T \,d\nu(s)\right\}\cdot
\int_t^T \exp\left\{-r(T-s)+\int_s^T \,d\nu(\tau)\right\}\,d\mu(s)
\end{align}
is a nonnegative monotone decreasing function defined on $[0,T]$, and the problem now read as follows:
\begin{equation}
\label{eq:I01a}
u_t+\tfrac{1}{2}\left(x-b(t)\right)^2\sigma^2 u_{xx}=0.
\end{equation}
\begin{equation}
\label{eq:I02b}
u(T,x)=(x-K)_{+}.
\end{equation}
It is mathematically natural to investigate existence, uniqueness, and regularity of a solution to the degenerate parabolic problem \eqref{eq:I01a}, \eqref{eq:I02b}.
The existence and uniqueness question can be easily answered by using a probabilistic argument.
Indeed, the problem \eqref{eq:I01a}, \eqref{eq:I02b} admits a unique probabilistic solution
\begin{equation}
\label{eq:I09u}
u(t,x):=\bE f(X_{T}(t,x)),
\end{equation}
where $f(y):=(y-K)_{+}$ and $X_s=X_s(t,x)$ is the stochastic process that satisfies, for $t\in[0,T]$ and $x\in\bR$, the following SDE:
\begin{equation}
\label{eq:I08z}
\left\{\begin{array}{l}
dX_s=(X_s-b_s)\sigma\,dw_s,\quad s\ge t, \quad (\,b_s=b(s)\,)\\
X_{t}=x.
\end{array}\right.
\end{equation}
On the other hand, regularity of the probabilistic solution $u$ defined in
\eqref{eq:I09u} is a  more subtle issue.
There is a classical result, originally due to Freidlin, saying that
if $f$ in \eqref{eq:I09u} is twice continuously differentiable and satisfies a certain growth condition,
then $u(t,x)$ defined  by \eqref{eq:I09u} is meaningful
and twice differentiable with respect to $x$ continuously in $(t,x)$, etc.;
see e.g. \cite[Theorem V.7.4]{Krylov}.
However, in our case, $f$ is only Lipschitz continuous and thus
the above mentioned result is not applicable.
As a matter of fact, it is not trivial whether or not the problem
\eqref{eq:I03c}, \eqref{eq:I04d} admits any classical or strong solution.
This regularity question was studied by one of the authors in \cite{Kim}.
It is shown in \cite{Kim} that if $K=0$ (in this case we have the fixed strike Asian call option) and if $b(t)$ is a monotone decreasing
Lipschitz continuous function, then the probabilistic solution $u$ defined
in \eqref{eq:I09u} is indeed a classical solution.
We note that $b(t)$ satisfies such an assumption if $d\mu(t)=\rho(t)\,dt$ for some $\rho\in L^\infty([0,T])$ satisfying $\rho(t)\geq \rho_0>0$; i.e.  the measure $\mu(t)$ is absolutely continuous with respect to the Lebesgue measure $dt$ and its density function $\rho(t)$ is strictly positive and bounded.
This excludes the cases when $\rho(t)$ is a nonnegative step function that vanishes on some intervals or when $\mu(t)$ is a linear combination of Dirac delta functions, which corresponds to discretely sampled Asian options.
These two cases are important in practice but have been left out in \cite{Kim}.

The goal of this article is, roughly speaking, to show that even in those cases, the probabilistic solution of problem \eqref{eq:I03c}, \eqref{eq:I04d} is still a classical solution.
As a matter of fact, we even give an improvement to the main result of \cite{Kim}.
To be precise, we will assume that the function $b(t)$ has the following properties.
\begin{enumerate}[i)]
\item
$b(t)$ is nonnegative and monotone decreasing on $[0,T]$ 
\item
$b(t)$ is discontinuous at most finitely many points $t_1<\cdots < t_n$ in $[0,T]$.
\item
$b(T)=0$ and there is an $\epsilon>0$ such that
\begin{enumerate}
\item
 $b(t)=0$ on $[T-\epsilon, T]$ if $K \neq 0$.
\item
$m_1 \le -b'(t) \le m_2$ a.e. on $[T-\epsilon, T]$ for some $m_1, m_2 >0$  if $K=0$.
\end{enumerate}
\end{enumerate}
It is clear that condition ii) allows us to treat the discrete sampling case.
We point out that in \cite{Kim} it is assumed that $K=0=b(T)$ and $m_1 \le -b'(t) \le m_2$ a.e. on $[0,T]$, so that $b(t)$ in \cite{Kim} satisfies the above properties.
We also note that the condition iii) is technical but is a generic one in the sense that it can be always realized in practice by perturbing sampling strategy.
In particular, note that by \eqref{eq:I04x}, we have $b(T)=0$ unless the measure $\mu(t)$ has a point mass at $T$.
We use the notation 
\begin{equation*}
\bH_T:= [0,T)\times \bR,\quad
\bar{\bH}_T:= [0,T]\times \bR,\quad
\mathring{\bH}_T:=\bH_T\setminus \left(\bigcup_{i=1}^n \set{t_i}\times \bR\right)
\end{equation*}
in our main result stated below.

\begin{theorem}
\label{thm:I01}
Let $b(t)$ satisfy the conditions i) - iii) above.
Suppose $u$ is the probabilistic solution of the problem \eqref{eq:I01a},
\eqref{eq:I02b};
i.e., $u(t,x)$ is defined by \eqref{eq:I09u}.
Then $u(t,x)$ is continuous in $\bar{\bH}_T$ and satisfies the terminal
condition \eqref{eq:I02b}.
Moreover, $u(t,x)$ is twice differentiable with respect
to $x$ continuously in $\bH_T$, differentiable with respect to $t$ continuously
in $\mathring{\bH}_T$, and satisfies the equation \eqref{eq:I01a} in $\mathring{\bH}_T$.
\end{theorem}

It is clear from \eqref{eq:I01a} that $u_t$ cannot be continuous where $u_{xx}$ is continuous but $b(t)$ is not continuous.
Therefore, if  the measure $\mu(t)$ contains a pure point mass (i.e. a discrete sampling), then $u_t$ cannot be continuous in the entire $\bH_T$.

\begin{conclusion}
The case $K=0$ corresponds to the fixed strike Asian call option. 
In that case, it is recommended to employ a continuous sampling near the terminal time $T$.
Except for the case $K=0$, it is recommended not to sample near the terminal time to ensure that the solution becomes a classical one.
\end{conclusion}

A couple of further remarks are in order.

\begin{remark}
Suppose $K=b(T) \neq 0$ and let $T'=\inf\set{t\in [0,T]:b(t)=K}$.
It is a matter of computation to verify that probabilistic solution
$u$ of the problem \eqref{eq:I01a}, \eqref{eq:I02b} in $[T',T]\times \bR$ is
given by $u(t,x)=(x-K)_{+}$.
Of course, $u$ is not twice continuously differentiable with respect to $x$
there.
\end{remark}

\begin{remark}
In Ve\v{c}e\v{r}'s PDE method, the price of Asian option is determined by
$u(0,\cdot)$.
Theorem~\ref{thm:I01} suggests that to minimize numerical error in computing
$u(0,\cdot)$ by finite difference methods, one should include in time grids all the points where $b(t)$ is discontinuous (i.e., discrete sampling points).
\end{remark}

%------------------------------------------------------------------------------%
\section{Proof of Theorem~\ref{thm:I01}}
%------------------------------------------------------------------------------%
For $(t,x) \in \bar{\bH}_T$, let $X_s=X_s(t,x)$ be the stochastic process
which satisfies \eqref{eq:I08z}.
It is well known that such a process $X_s$ exists; see e.g.,
\cite[Theorem V.1.1]{Krylov}.
The probabilistic solution $u$ of the
problem \eqref{eq:I01a}, \eqref{eq:I02b} is then given by
the formula \eqref{eq:I09u}.
It is then evident that $u$ is continuous in $\bar{\bH}_T$ and satisfies
the terminal condition \eqref{eq:I02b}.
If $f$ in \eqref{eq:I09u} were twice continuously differentiable,
then, as it was mentioned in the introduction,
the theorem would follow from \cite[Theorem V.7.4]{Krylov}.
But this is clearly not the case since $f(y)=(y-K)_{+}$ is merely a Lipschitz
continuous function.
However, it should be pointed out that $u$ also has the representation
\[
u(t,x)=\bE u(X_{\tilde T}(t,x)),\quad \forall \tilde T\in [t,T].
\]
Therefore, if $u(\tilde T,\cdot)$ is twice continuously differentiable for some
$\tilde T\in (0,T)$, then we would have the second part of the theorem with
$T$ replaced by $\tilde T$; i.e., $u(t,x)$ is twice
differentiable with respect to $x$ continuously in $\bH_{\tilde T}$, differentiable
with respect to $t$ continuously in $\mathring{\bH}_{\tilde T}$, and satisfies the equation
\eqref{eq:I01a} in $\mathring{\bH}_{\tilde T}$.

Therefore, in the case when $K=0$, we have $m_1 \le -b'(t) \le m_2$ a.e. on $[T-\epsilon, T]$ for some $m_1, m_2>0$, and thus by \cite[Theorem~1.10]{Kim}, we are done.
It only remains to consider the case when $K\neq 0$ and $b(t)=0$ on $[T-\epsilon, T]$.
By the above observation, we have the theorem if we show that
$u\in \Co_{t,x;\; loc}([T-\epsilon,T)\times\bR)$ and satisfies the equation \eqref{eq:I01a}
there.
Therefore, it will be enough for us to prove that the probabilistic solution $u$ of the problem
\begin{align}
\label{eq:P01a}
u_t+\frac{1}{2} \sigma^2 x^2 u_{xx}=0 \quad\text{in}\quad \bH_T,\\
\label{eq:P01b}
u(T,x)=(x-K)_{+}
\end{align}
is a classical solution.
By the Black-Scholes-Merton's options pricing formula, we know its solution is $C^\infty$ in both $x$ and $t$ in $\bH_T$ and thus we are done.
\hfill\qedsymbol

\section{Appendix}
We give a self-contained proof that the probabilistic solution $u(t,x)$ of the problem \eqref{eq:P01a}, \eqref{eq:P01b} is a classical solution without invoking Black-Scholes-Merton's formula.
Let $Y_s=Y_s(x)$ be the process satisfying the stochastic equation
\begin{equation}
\label{eq3.37}
dY_s=\sigma Y_s\,dw_s,\quad Y_0=x.
\end{equation}
It is easy to verify that $Y_s$ is given by
$Y_s= x e^{\sigma w_s-s/2}$.
Then $u(t,x)$ is given by
\begin{equation}
\label{eq:M02a}
u(t,x)=\bE(Y_{T-t}(x)-K)_{+}=\bE\left(xe^{\sigma w_{T-t}-(T-t)/2}-K\right)_{+}.
\end{equation}
Since $Y_s$ is a martingale and $(y-K)_{+}=(y-K)+(K-y)_{+}$, we also have
\begin{equation}
\label{eq:M02b}
u(t,x)=x-K+\bE\left(K-xe^{\sigma w_{T-t}-(T-t)/2}\right)_{+}.
\end{equation}
The above computations lead us to define
\begin{equation}
\label{eq:M11f}
v(t,x):=
\left\{\begin{array}{l}
\bE\left(xe^{\sigma w_{T-t}-(T-t)/2}-K\right)_{+} \quad\text{if}\quad K>0,\\
\bE\left(K-xe^{\sigma w_{T-t}-(T-t)/2}\right)_{+} \quad\text{if}\quad K<0.
\end{array}\right.
\end{equation}
By \eqref{eq:M02a} and \eqref{eq:M02b}, we find that
$u\in\Co_{t,x;\; loc}(\bH_T)$ and satisfies \eqref{eq:P01a} if and only if
$v\in\Co_{t,x;\; loc}(\bH_T)$ and satisfies \eqref{eq:P01a}.
Denote
\begin{equation*}
\bH_T^{+}:=[0,T)\times(0,\infty),\quad
\bH_T^{-}:=[0,T)\times(-\infty,0).
\end{equation*}
It is easy to check that $v\equiv 0$ in $\bH_T\setminus \bH_T^{+}$ if $K>0$ and
$v\equiv 0$ in $\bH_T\setminus\bH_T^{-}$ if $K<0$.
Also, by using an approximation argument similar to the one used
in \cite{Kim}, we find that $v$ belongs to $\Ho_{t,x, loc}(\bH_T^{+}\cup \bH_T^{-})$ and satisfies the equation \eqref{eq:P01a} in $\bH_T^{+}\cup \bH_T^{-}$
regardless the sign of $K$.
Therefore, the proof will be complete if we show
\begin{equation}
\label{eq4.10}
\lim_{x\to 0}\,(|v_x|+|v_{xx}|+|v_t|)(t,x)=0\quad
\text{locally uniformly in }t\in [0,T).
\end{equation}

\begin{lemma}
Let $v$ be defined by \eqref{eq:M11f} and denote
\begin{equation*}
Q:=
\left\{
\begin{array}{l}
{[0,T)\times (0,K)}
\quad\text{if}\quad K>0, \\
{[0,T)\times (K,0)}
\quad\text{if}\quad K<0.
\end{array}\right.
\end{equation*}
Then we have the following estimate for $v(t,x)$ in $Q$:
\begin{equation}
\label{eq:M07g}
0\leq v(t,x)\leq \sqrt{\frac{2}{\pi}} \frac{\sigma K\sqrt{T}}{\ln |K/x|}
\exp\set{-\frac{(\ln|K/x|)^2}{\sigma^2 T}}.
\end{equation}
\end{lemma}

\begin{proof}
We shall assume that $K>0$ and carry out the proof.
The proof for the case when $K<0$ is parallel and shall be omitted.
First of all, it is clear from \eqref{eq:M11f} that $v\geq 0$.
For any $(t,x)\in Q$, we define the process $Z_s=Z_s(t,x)=(t+s,Y_s(x))$,
where $Y_s=Y_s(x)=xe^{\sigma w_s-s/2}$ satisfies the stochastic equation
\eqref{eq3.37} above.
Let $\tau=\tau(t,x)$ be the first exit time of $Z_s(t,x)$ from $Q$.
We define $\tilde v(t,x)$ by
\begin{equation*}
\tilde v(t,x)=\bE g(Z_\tau(t,x))=\bE g(t+\tau,Y_\tau(x)),
\end{equation*}
where the values of $g=g(s,y)$ on the parabolic boundary $\partial_p Q$ of
$Q$ are given by
\begin{equation*}
\left\{\begin{array}{l}
g(T,y)=0\quad\text{for}\quad 0<y<K,\\
g(s,0)=0\quad\text{and}\quad g(s,K)=K \quad\text{for}\quad 0\le s\le T.
\end{array}\right.
\end{equation*}
We claim that $v\leq \tilde v$ in $Q$.
To see this, first note that
\begin{equation*}
v(t,x)=\bE v(Z_\tau(t,x))=\bE v(t+\tau,Y_\tau(x)).
\end{equation*}
Thus, it is enough to show that $v\leq g$ on $\partial_p Q$.
It is obvious that $v(T,y)=0$ for any $y\in (0,K)$ and $v(t,0)=0$ for
any $t\in [0,T]$.
Also, since $(y-K)_{+}\leq y$ for any $y\geq 0$ and $e^{\sigma w_s-s/2}$
is a martingale, we have
\begin{equation*}
v(t,K) \leq \bE\left(Ke^{\sigma w_{T-t}-(T-t)/2}\right)= K,\quad
\forall t\in [0,T].
\end{equation*}
It thus follows that $v\leq \tilde v$ in $Q$.
Therefore, we have
\begin{align*}
v(t,x)
&\leq \tilde v(t,x)=K\bP\{xe^{\sigma w_\tau-\tau/2}=K\}\\
&\le K\bP\set{\sup_{0\le s<T-t}(\sigma w_s-s/2)\ge \ln (K/x)}\\
&\le K\bP\set{\sup_{0\le s<T}(\sigma w_s-s/2)\ge \ln (K/x)}\\
&\le K\bP\set{\sup_{0\le s<T}w_s\ge \sigma^{-1}\ln (K/x)}
=2K\bP\set{w_T\ge \sigma^{-1}\ln (K/x)}\\
&\le K\sqrt{\frac{2}{\pi}} \frac{\sigma\sqrt{T}}{\ln (K/x)}
\exp\set{-\frac{(\ln(K/x))^2}{\sigma^2 T}},
\end{align*}
where, in the last step, we have used an inequality
\begin{align*}
\int_\alpha^\infty e^{-x^2/2}\,dx &\leq
\frac{1}{\alpha}\int_\alpha^\infty x e^{-x^2/2}\,dx
=\alpha^{-1} e^{-\alpha^2/2},\quad \forall\alpha>0.
\end{align*}
The lemma is proved.
\end{proof}

Now, we prove the statement \eqref{eq4.10}.
We shall assume that $K>0$ since the case when $K<0$ can be treated in a similar way.
We extend $v$ to zero on $(T,\infty)\times (0,K)$.
Then, it is easy to see that $v$ belongs to $\Ho_{t,x;\; loc}([0,\infty)\times (0,K))$ and satisfies both the equation \eqref{eq:P01a} and the estimate \eqref{eq:M07g} in $[0,\infty)\times (0,K)$.
Denote
\begin{equation*}
Q_{r}(t_0,x_0):=(t_0,t_0+1)\times(x_0-r,x_0+r),\quad
\Pi_\rho:=(0,\rho^2)\times (-\rho,\rho).
\end{equation*}
Let $(t_0,x_0)\in [0,T)\times (0,2K/3)$ and set $r=r(x_0):=x_0/2$.
Define
\begin{equation*}
V(t,x)=v(t_0+ t,x_0+rx)
\end{equation*}
It is easy to verify that $V(t,x)$ satisfies the equation
\begin{equation*}
V_t+\frac{1}{2} a(x) V_{xx}=0\quad\text{in}\quad \Pi_1,
\end{equation*}
where $a(x):= \sigma^2 r^{-2}(x_0+rx)^2$ satisfies
\begin{equation*}
\sigma^2 \leq a(x)\le 9\sigma^2 \quad\text{and}\quad |a'(x)|\leq 6\sigma^2, \quad \forall x\in (-1,1).
\end{equation*}
By the Schauder estimates (see e.g., \cite{Kr2}), there is some $C=C(\sigma)$ such that
\[
|V_x(0,0)|+|V_{xx}(0,0)|+|V_t(0,0)|\leq C \sup_{\Pi_1}|V|,
\]
while by the estimate \eqref{eq:M07g} we have
\[
\sup_{\Pi_1} |V|=\sup_{Q_{r}(t_0,x_0)} |v| \leq
\frac{\sigma K\sqrt{2T}}{\sqrt{\pi}\ln |K/3r|}
\exp\set{-\frac{(\ln|K/3r|)^2}{\sigma^2 T}}.
\]
Hence, there is some $r_0=r_0(K)$ and $N=N(\sigma, T, K)$ such that
\begin{equation}
\label{eq:P19d}
 |V_x(0,0)|+ |V_{xx}(0,0)|+|V_t(0,0)|\leq N e^{-(\ln r)^2/N},
\end{equation}
provided that $r<r_0$.
Note that \eqref{eq:P19d} translates to as follows:
There is some $r_0=r_0(K)$ and $N=N(\sigma, T, K)$ such that
if $x_0<r_0$ then
\[
x_0 |v_x(t_0,x_0)|+ x_0^2 |v_{xx}(t_0,x_0)|+|v_t(t_0,x_0)|\leq N e^{-(\ln x_0)^2/N}.
\]
The above estimate obviously implies \eqref{eq4.10}.
\hfill\qedsymbol

\begin{acknowledgment}
The authors thank Jan Ve\v{c}e\v{r} for very helpful discussion.
\end{acknowledgment}

%-------------------------------------------------%

\end{document}